\documentclass{article}  
\usepackage[utf8]{inputenc}
\usepackage[height=8.4in,width=6in]{geometry}
\usepackage[sort&compress]{natbib}
\usepackage{graphicx}
\usepackage{subfigure}

\usepackage{algorithm}
\usepackage{algpseudocode}

\usepackage{hyperref}

\usepackage{amsmath,amssymb,amsfonts,amsthm}
\usepackage{mathrsfs,bm}
\usepackage{multicol}
\usepackage{multirow}
\usepackage{fullpage}
\usepackage{paralist}
\usepackage{stmaryrd} 
\usepackage{xspace}
\usepackage{array}

\bibpunct[; ]{[}{]}{,}{a}{}{;}

\newcommand{\Xc}{\mathcal{X}}
\newcommand{\Mc}{\mathcal{M}}
\newcommand{\Oc}{\mathcal{O}}
\newcommand{\eh}{\hat{E}}

\newcommand{\reals}{\mathbb{R}}

\newcommand{\sx}{{x^*}}
\newcommand{\sml}[1]{{\small #1}}
\DeclareMathOperator{\prox}{prox}

\newcommand{\ind}[2]{\delta({#2} | {#1})}

\newcommand{\proj}{\Pi}
\newcommand{\pp}{\mathscr{P}}
\newcommand{\plg}{\pp_\eta^g}

\newcommand{\half}{\tfrac{1}{2}}
\newcommand{\set}[1]{\left\{ {#1}\right\}}
\newcommand{\ip}[2]{\langle {#1},\, {#2} \rangle}
\newcommand{\nlsum}{\sum\nolimits}
\newcommand{\nlmin}{\min\nolimits}

\newcommand{\pnorm}[2]{\| {#1} \|_{#2}}
\newcommand{\norm}[1]{\|{#1}\|}
\newcommand{\algo}{\textsc{NIPS}\xspace}

\newcommand{\bignorm}[2]{\left\| {#1} \right\|_{#2}}

\newcommand{\norml}[1]{\pnorm{#1}{1}}

\newcommand{\frob}[1]{\|{#1}\|_{\text{F}}}

\newcommand{\fromto}[3]{\sml{$#1 \le #2 \le #3$}}
\DeclareMathOperator*{\argmin}{argmin}

\DeclareMathOperator*{\id}{Id}

\newtheorem{theorem}{Theorem}
\newtheorem{lemma}[theorem]{Lemma}
\newtheorem{corr}[theorem]{Corollary}
\theoremstyle{definition}
\newtheorem{defn}[theorem]{Definition}

\begin{document}

\title{Nonconvex proximal-splitting: batch and incremental algorithms}
\author{Suvrit Sra\thanks{Max Planck Institute for Intelligent Systems, T\"ubingen, Germany. Email: {\it suvrit@tuebingen.mpg.de}}}

\maketitle

\begin{abstract}
  We study a class of large-scale, nonsmooth, and nonconvex optimization problems. In particular, we focus on nonconvex problems with \emph{composite} objectives. This class  includes the extensively studied convex, composite objective problems as a subclass. To solve composite nonconvex problems we introduce a powerful new framework based on asymptotically \emph{nonvanishing} errors, avoiding the common stronger assumption of  vanishing errors. Within our new framework we derive both batch and incremental proximal splitting algorithms. To our knowledge, our work is first to develop and analyze incremental \emph{nonconvex} proximal-splitting algorithms, even if we were to disregard the ability to handle nonvanishing errors. We illustrate one instance of our general framework by showing an application to large-scale nonsmooth matrix factorization.
\end{abstract}

\section{Introduction}
This paper focuses on nonconvex \emph{composite objective} problems having the form
\begin{equation}
  \label{eq.1}
  \text{minimize}\quad \Phi(x) := f(x) + h(x)\qquad x \in \Xc,
\end{equation}
where $f : \reals^n \to \reals$ is continuously differentiable and possibly nonconvex, $h : \reals^m \to \reals \cup \set{\infty}$ is lower semi-continuous (lsc) and convex (possibly nonsmooth), and $\Xc$ is a compact convex set. We also make the common assumption that $f$ has a Lipschitz continuous gradient within $\Xc$, written as $f \in C_L^1(\Xc)$; that is, there is a constant $L> 0$ such that
\begin{equation}
  \label{eq.lipf}
  \norm{\nabla f(x) - \nabla f(y)} \le L\norm{x-y}\qquad\text{for all}\quad x, y \in \Xc.
\end{equation}

Problem~\eqref{eq.1} is a natural but far-reaching generalization of \emph{composite objective} convex problems, which continue to enjoy tremendous importance in machine learning; see e.g.,~\citep{duchi09a,bach11,lxiao09,beck09}. Although, convex formulations are extremely useful, for many difficult problems a nonconvex formulation is more natural. Familiar examples include matrix factorization~\citep{mairal10a,lee00}, blind deconvolution~\citep{kundur96}, dictionary learning~\citep{delgado03,mairal10a}, and neural networks~\citep{bertsekas99,haykin94}. 

The primary contribution of this paper is theoretical and algorithmic. Specifically, we present a new framework: \textsc{N}onconvex \textsc{I}nexact \textsc{P}roximal \textsc{S}plitting (\algo). Our framework  solves~\eqref{eq.1} by ``splitting'' the task into smooth (gradient) and nonsmooth (proximal) parts. Beyond splitting, the most notable feature of \algo is that it allows \emph{computational errors}. This capability proves critical to obtaining a scalable, incremental-gradient variant of \algo, which, to our knowledge, is the first incremental nonconvex proximal-splitting method in the literature.

A further distinction of \algo lies in how it models computational errors. Notably, \algo \emph{does not} require the errors to vanish in the limit, a realistic assumption as often one has limited to no control over computational errors inherent to a complex system. In accord with the errors, \algo also \emph{does not} require stepsizes (learning rates) to shrink to zero. In contrast, most incremental-gradient methods~\citep{bertsekas10} and stochastic gradient algorithms~\citep{gaivoronski94} \emph{do assume} that the computational errors and stepsizes decay to zero. We do not make these simplifying assumptions, which complicates the convergence analysis but results perhaps in a more satisfying description. 

Our analysis builds on the remarkable work of~\citet{solodov97}, who studied the simpler setting of \emph{differentiable} nonconvex problems (corresponding to the choice $h \equiv 0$ in~\eqref{eq.1}). \algo is strictly more general: unlike~\citep{solodov97} it solves a \emph{non-differentiable} problem by allowing a nonsmooth regularizer $h\not\equiv 0$, which it tackles by invoking the fruitful idea of \emph{proximal-splitting}~\citep{combettes10}. 


Proximal splitting has proved to be exceptionally effective and practical~\citep{combettes10,beck09,bach11,duchi09a}. It retains the simplicity of gradient-projection while handling the nonsmooth regularizer $h$ via its proximity operator. This style is especially attractive because for several important choices of $h$, efficient implementations of the associated proximity operators exist~\citep{bach11,liu09a,mairal10a}. For convex problems, an alternative to proximal splitting is the subgradient method; similarly, for nonconvex problems one could use a generalized subgradient method~\citep{clarke83,ermoliev98}. However, as in the convex case, the use of subgradients has drawbacks: it fails to exploit the composite structure, and even when using sparsity promoting regularizers it does not generate intermediate sparse iterates~\citep{duchi09a}.

Among batch nonconvex splitting methods, an early paper is~\citep{fukushima81}. More recently, in his pioneering paper on convex composite minimization, \citet{nesterov07} also briefly discussed nonconvex problems. Both~\citep{fukushima81} and~\citep{nesterov07}, however, enforced monotonic descent in the objective value to ensure convergence. Very recently,~\citet{attbosv11} have introduced a generic method for nonconvex nonsmooth problems based on Kurdyka-\L{}ojasiewicz theory, but their entire framework too hinges on descent.

This insistence on descent makes these methods unsuitable for incremental, stochastic, or online variants, all of which usually lead to a nonmonotone sequence of objective values. Among nonmonotonic methods that apply to~\eqref{eq.1}, we are aware of the generalized gradient-type algorithms of~\citep{solodov98} and the stochastic generalized gradient methods of~\citep{ermoliev98}. Both methods, however, are analogous to the usual subgradient-based algorithms, and fail to exploit the composite objective structure.

But despite its desirability and potential benefits, proximal-splitting for exploiting composite objectives does not apply out-of-the-box to~\eqref{eq.1}: nonconvexity raises significant obstructions, especially because nonmonotonic descent in the objective function values is allowed. Overcoming these obstructions to achieve a scalable non-descent based method is what makes the \algo framework novel. 

\section{The \algo Framework}
To simplify presentation, we replace $h$ by the penalty function
\begin{equation}
  \label{eq.4}
  g(x) := h(x) + \ind{\Xc}{x},
\end{equation}
where $\ind{\Xc}{\cdot}$ is the \emph{indicator function} for $\Xc$:  $\ind{\Xc}{x}=0$ for $x \in \Xc$, and $\ind{\Xc}{x}=\infty$ for $x \not\in \Xc$. With this notation, we may rewrite~\eqref{eq.1} as the \emph{unconstrained} problem:
\begin{equation}
  \label{eq.8}
  \nlmin_{x \in \reals^n}\quad \Phi(x) := f(x) + g(x),
\end{equation}
and this particular formulation is our primary focus. We solve~\eqref{eq.8} via a proximal-splitting approach, so let us begin by defining our most important component.
\begin{defn}[Proximity operator]
  \label{def.prox}
  Let $g : \reals^n \to \reals$ be an lsc, convex function. The \emph{proximity operator} for $g$, indexed by $\eta > 0$, is the nonlinear map~\citep[see e.g.,][Def.~1.22]{rockafellar98}:
  \begin{equation}
    \label{eq.7}
    \pp_\eta^g\ :\quad y\mapsto\ \argmin_{x \in \reals^n}\ \bigl(g(x) + \tfrac{1}{2\eta}\norm{x-y}^2\bigr).
  \end{equation}
\end{defn}

The operator~\eqref{eq.7} was introduced by~\citet{moreau62} (1962) as a generalization of orthogonal projections. It is also key to Rockafellar's classic \emph{proximal point algorithm}~\citep{rock76}, and it arises in a host of \emph{proximal-splitting} methods~\citep{combettes10,duchi09a,bach11,beck09}, most notably in \emph{forward-backward splitting} (FBS)~\citep{combettes10}.


FBS is particularly attractive because of its simplicity and algorithmic structure. It minimizes convex composite objective functions by alternating between  ``forward'' (gradient) steps and ``backward'' (proximal) steps. Formally, suppose $f$ in~\eqref{eq.8} is convex; for such $f$, FBS performs the iteration
\begin{equation}
  \label{eq.59}
  x^{k+1} = \pp_{\eta_k}^g(x^k - \eta_k\nabla f(x^k)),\quad k=0,1,\ldots,
\end{equation}
where $\set{\eta_k}$ is a suitable sequence of stepsizes. The usual convergence analysis of FBS is intimately tied to convexity of $f$. Therefore, to tackle nonconvex $f$ we must take a different approach. As previously mentioned, such approaches were considered by~\citet{fukushima81} and~\citet{nesterov07}, but both proved convergence by enforcing monotonic descent. 

This insistence on descent severely impedes scalability. Thus, the key challenge is: \emph{how to retain the algorithmic simplicity of FBS and allow nonconvex losses, without sacrificing scalability?}  

We address this challenge by introducing the following \emph{inexact} proximal-splitting iteration:
\begin{equation}
  \label{eq.main}
  x^{k+1} = \pp_{\eta_k}^g(x^k - \eta_k\nabla f(x^k) + \eta_ke(x^k)),\quad k=0,1,\ldots,
\end{equation}
where $e(x^k)$ models the \emph{computational errors} in computing the gradient $\nabla f(x^k)$. We also assume that for $\eta > 0$ smaller than some stepsize $\bar \eta$, the computational error is uniformly \emph{bounded}, that is,
\begin{equation}
  \label{eq.60}
  \eta\norm{e(x)} \le \bar{\epsilon},\quad\text{for some fixed error level }\ \bar{\epsilon} \ge 0,\quad\text{and } \forall x \in \Xc.
\end{equation}
Condition~\eqref{eq.60} is weaker than the typical vanishing error requirements
\begin{equation*}
  \nlsum_k \eta\norm{e(x^k)} < \infty,\qquad \lim_{k\to \infty} \eta\norm{e(x^k)} = 0,
\end{equation*}
which are stipulated by most analyses of methods with gradient errors~\citep{bertsekas99,bertsekas10}. Obviously, since errors are nonvanishing,  exact stationarity cannot be guaranteed. We will, however, show that the iterates produced by~\eqref{eq.main} do progress towards reasonable \emph{inexact stationary points}. We note in passing that even if we assume the simpler case of vanishing errors, \algo is still the first nonconvex proximal-splitting framework that does not insist on monotonicity, which which complicates convergence analysis but ultimately proves crucial to scalability. 


\begin{algorithm}[htbp]\small
  \begin{algorithmic}
    \Require{Operator $\pp_\eta^g$, and a sequence $\set{\eta_k}$ satisfying
      \begin{equation}
        \label{eq.38}
        c \le \liminf\nolimits_k \eta_k,\quad \limsup\nolimits_k \eta_k \le \min\set{1, 2/L-c},\quad0 < c < 1/L.
      \end{equation}}
    \Ensure{Approximate solution to~\eqref{eq.main}}
    \hrule
    \State $k \gets 0$; Select arbitrary $x^0 \in \Xc$
    \While{$\neg$ converged}
    \State Compute approximate gradient $\widetilde{\nabla}f(x^k) := \nabla f(x^k) - e(x^k)$
    \State Update: $x^{k+1} = \pp_{\eta_k}^g(x^k - \eta_k\widetilde{\nabla}f(x^k))$
    \State $k \gets k + 1$
    \EndWhile
  \end{algorithmic}
  \caption{Inexact Nonconvex Proximal Splitting (\algo)}\label{alg.nocops}
\end{algorithm}

\subsection{Convergence analysis}
\label{sec.convergence}

We begin by characterizing inexact stationarity. A point $x^*$ is a stationary point for~(\ref{eq.8}) if and only if it satisfies the \emph{inclusion}
\begin{equation}
  \label{eq.30}
  0 \in \partial_C\Phi(x^*) := \nabla f(x^*) + \partial g(x^*),
\end{equation}
where $\partial_C\phi$ denotes the Clarke subdifferential~\citep{clarke83}. A brief exercise shows that this inclusion may be equivalently recast as the fixed-point equation (which augurs the idea of proximal-splitting)
\begin{equation}
  \label{eq.9}
  x^* = \pp_\eta^g(x^* - \eta\nabla f(x^*)),\quad\text{for } \eta > 0.
\end{equation}
This equation helps us define a measure of inexact stationarity: the \emph{proximal residual}
\begin{equation}
  \label{eq.5}
  \rho(x) := x - \pp_1^g(x - \nabla f(x)).
\end{equation}  
Note that for an exact stationary point $x^*$ the residual norm $\norm{\rho(x^*)}=0$. Thus, we call a point $x$ to be $\epsilon$-\emph{stationary} if for a prescribed error level $\epsilon(x)$, the corresponding residual norm satisfies
\begin{equation}
  \label{eq.10}
  \norm{\rho(x)} \le \epsilon(x).
\end{equation}
Assuming the error-level $\epsilon(x)$ (say if $\bar{\epsilon}=\lim\sup_k \epsilon(x^k)$) satisfies the bound~\eqref{eq.60}, we prove below that the iterates $\set{x^k}$ generated by~(\ref{eq.main}) satisfy an approximate stationarity condition of the form~\eqref{eq.10}, by allowing the stepsize $\eta$ to become correspondingly small (but strictly bounded away from zero). 

We start by recalling two basic facts, stated without proof as they are standard knowledge.
\begin{lemma}[Lipschitz-descent~\protect{\citep[see e.g.,][Lemma~2.1.3]{nesterov04}}]
  \label{lem.lip}
  Let $f \in C_L^1(\Xc)$. Then,
  \begin{equation}
    \label{eq.11}
    |f(x) - f(y) - \ip{\nabla f(y)}{x-y}| \le \tfrac{L}{2}\norm{x-y}^2,\quad \forall\ x, y \in \Xc.
  \end{equation}
\end{lemma}

\begin{lemma}[Nonexpansivity~\protect{\citep[see e.g.,][Lemma~2.4]{combettes05}}]
  \label{lem.nex}
  The operator $\pp_\eta^g$ is \emph{nonexpansive}, that is,
  \begin{align}
    \label{ne}
    \norm{\plg(x) - \plg(y)} \le \norm{x-y},\quad\forall\ x, y \in \reals^n.
  \end{align}
\end{lemma}


Next we prove a crucial monotonicity property that actually subsumes similar results for projection operators derived by~\citet[Lem.~1]{gafni84}, and may therefore be of independent interest.
\begin{lemma}[Prox-Monotonicity]
  \label{lem.mono}
  Let $y, z \in \reals^n$, and $\eta > 0$. Define the functions
  \begin{equation}
    \label{eq.19}
    p_g(\eta)\quad:=\quad \tfrac{1}{\eta}\norm{\pp_\eta^g(y - \eta z) -y},\quad\text{and}\quad
    q_g(\eta)\quad:=\quad\norm{\pp_\eta^g(y - \eta z) - y}.
  \end{equation}
  Then, $p_g(\eta)$ is a decreasing function of $\eta$, and $q_g(\eta)$ an increasing function of $\eta$.
\end{lemma}
\begin{proof}
  Our proof exploits properties of Moreau-envelopes~\citep[pp.~19,52]{rockafellar98}, and we present it in the language of proximity operators. Consider the ``deflected''  proximal objective
  \begin{equation}
    \label{eq.18}
    m_g(x, \eta; y, z) := \ip{z}{x - y} + \tfrac{1}{2\eta}\norm{x-y}^2 + g(x),\quad\text{for some}\ y, z \in \Xc.
  \end{equation}
  Associate to objective $m_g$ the \emph{deflected Moreau-envelope}
  \begin{equation}
    \label{env}
    E_g(\eta) := \inf_{x \in \Xc}\ m_g(x, \eta; y, z),
  \end{equation}
  whose infimum is attained at the unique point $\pp_{\eta}^g(y-\eta z)$. Thus, $E_g(\eta)$ is differentiable, and its derivative is given by $E_g'(\eta) = -\tfrac{1}{2\eta^2}\norm{\pp_{\eta}^g(y-\eta z)-y}^2 = -\half p(\eta)^2$.
  Since $E_g$ is convex in $\eta$, $E_g'$ is increasing (\citep[Thm.~2.26]{rockafellar98}), or equivalently $p(\eta)$ is decreasing.  Similarly, define $\eh_g(\gamma) := E_g(1/\gamma)$; this function is concave in $\gamma$ as it is a pointwise infimum (indexed by $x$) of functions linear in $\gamma$~\citep[see e.g.,~\S3.2.3 in][]{boyd04}. Thus, its derivative $\eh_g'(\gamma) = \half\norm{\pp_{1/\gamma}^g(x- \gamma^{-1}y)-x}^2 = q_g(1/\gamma)$,  is a decreasing function of $\gamma$. Set $\eta=1/\gamma$ to conclude the argument about $q_g(\eta)$.
\end{proof}

We now proceed to bound the difference between objective function values from  iteration $k$ to $k+1$, by developing a bound of the form
\begin{equation}
  \label{eq.6}
  \Phi(x^k) - \Phi(x^{k+1}) \ge h(x^k).
\end{equation}
Obviously, since we do \emph{not} enforce strict descent, $h(x^k)$ may be negative too. However, we show that for sufficiently large $k$ the algorithm makes enough progress to ensure convergence.


\begin{lemma}
  \label{lem.descent}
  Let $x^{k+1}$, $x^k$, $\eta_k$, and $\Xc$ be as in~(\ref{eq.main}), and that $\eta_k\norm{e(x^k)} \le \epsilon(x^k)$ holds. Then,
  \begin{equation}
    \label{eq.31}
    \Phi(x^k) - \Phi(x^{k+1})  \quad\ge\quad\tfrac{2-L\eta_k}{2\eta_k}\norm{x^{k+1}-x^k}^2 -
    \tfrac{1}{\eta_k}\epsilon(x^k)\norm{x^{k+1}-x^k}.
  \end{equation}
\end{lemma}
\begin{proof}
  For the deflected Moreau envelope~(\ref{eq.18}), consider the directional derivative $\text{d}m_g$ with respect to $x$ in the direction $w$; at $x=x^{k+1}$, this derivative satisfies the optimality condition
  \begin{equation}
    \label{eq.22}
    \text{d}m_g(x^{k+1}, \eta; y, z)(w) = \ip{z + \eta^{-1}(x^{k+1}-y) + s^{k+1}}{w} \ge 0,\quad s^{k+1} \in \partial g(x^{k+1}).
  \end{equation}
  Set $z=\nabla f(x^k) - e(x^k)$, $y=x^k$, and $w=x^k-x^{k+1}$ in~\eqref{eq.22}, and rearrange to obtain
  \begin{equation}
    \label{eq.35}
    \ip{\nabla f(x^k) - e(x^k)}{x^{k+1}-x^k}\le\ip{\eta^{-1}(x^{k+1}-x^k) + s^{k+1}}{x^k-x^{k+1}}.
  \end{equation}
  From Lemma~\ref{lem.lip} it follows that
  \begin{equation}
    \label{eq.34}
    \Phi(x^{k+1}) \le f(x^k) + \ip{\nabla f(x^k)}{x^{k+1}-x^k} + \tfrac{L}{2}\norm{x^{k+1}-x^k}^2 + g(x^{k+1}),
  \end{equation}
  whereby upon adding and subtracting $e(x^k)$, and then using~\eqref{eq.35} we further obtain
  \begin{align*}
    &f(x^k) + \ip{\nabla f(x^k) - e(x^k)}{x^{k+1}-x^k} + \tfrac{L}{2}\norm{x^{k+1}-x^k}^2 + g(x^{k+1}) + \ip{e(x^k)}{x^{k+1}-x^k}\\
    &\le f(x^k) + g(x^{k+1}) + \ip{s^{k+1}}{x^k-x^{k+1}} + \bigl(\tfrac{L}{2}-\tfrac{1}{\eta_k}\bigr)\norm{x^{k+1}-x^k}^2 + \ip{e(x^k)}{x^{k+1}-x^k}\\
    &\le f(x^k) + g(x^k) - \tfrac{2-L\eta_k}{2\eta_k}\norm{x^{k+1}-x^k}^2 + \ip{e(x^k)}{x^{k+1}-x^k}\\
    &\le \Phi(x^k) - \tfrac{2-L\eta_k}{2\eta_k}\norm{x^{k+1}-x^k}^2 + \norm{e(x^k)}\norm{x^{k+1}-x^k}\\
    &\le \Phi(x^k) - \tfrac{2-L\eta_k}{2\eta_k}\norm{x^{k+1}-x^k}^2 + \tfrac{1}{\eta_k}\epsilon(x^k)\norm{x^{k+1}-x^k}.
  \end{align*}
  The second inequality above follows from convexity of $g$, the third one from Cauchy-Schwarz, and the last one by assumption on $\epsilon(x^k)$. Now flip signs and apply~\eqref{eq.34} to conclude the bound~\eqref{eq.31}.
\end{proof}

Next we further bound~\eqref{eq.31} by deriving two-sided bounds on $\norm{x^{k+1}-x^k}$. 
\begin{lemma}
  \label{lem.xubounds}
  Let $x^{k+1}$, $x^k$, and $\epsilon(x^k)$ be as before; also let $c$ and $\eta_k$ satisfy~\eqref{eq.38}. Then,
  \begin{equation}
    \label{eq.xubounds}
    c\norm{\rho(x^k)} - \epsilon(x^k) \le \norm{x^{k+1}-x^k} \le \norm{\rho(x^k)} + \epsilon(x^k).
  \end{equation}
\end{lemma}
\begin{proof}
  First observe that from Lemma~\ref{lem.mono} that for $\eta_k > 0$ it holds that
  \begin{equation}
    \label{eq.37}
    \text{if}\ \ 1 \le \eta_k\ \ \text{then}\ \ q(1) \le q_g(\eta_k),\quad\text{and if}\ \  \eta_k \le 1\ \ \text{then}\ \ p_g(1) \le p_g(\eta_k) = \tfrac{1}{\eta_k}q_g(\eta_k).
  \end{equation}
  Using~\eqref{eq.37}, the triangle inequality, and Lemma~\ref{lem.nex}, we have
  \begin{alignat*}{2}
    \min\set{1,\eta_k}q_g(1) &= \min\set{1,\eta_k}\norm{\rho(x^k)}\quad\le\quad \norm{\pp_{\eta_k}^g(x^k-\eta_k\nabla f(x^k))-x^k}\\
    &\le\quad\norm{x^{k+1}-x^k} + \norm{x^{k+1}-\pp_{\eta_k}^g(x^k-\eta_k\nabla f(x^k))}\\
    &\le\quad\norm{x^{k+1}-x^k} + \norm{\eta_ke(x^k)}\quad\le\quad\norm{x^{k+1}-x^k} + \epsilon(x^k).
  \end{alignat*}
  From~\eqref{eq.38} it follows that for sufficiently large $k$ we have $\norm{x^{k+1}-x^k} \ge c\norm{\rho(x^k)} - \epsilon(x^k)$. For the upper bound note that
  \begin{align*}
    \norm{x^{k+1}-x^k} &\le \norm{x^k-\pp_{\eta_k}^g(x^k-\eta_k\nabla f(x^k))} + \norm{\pp_{\eta_k}^g(x^k-\eta_k\nabla f(x^k))-x^{k+1}}\\
    &\le \max\set{1,\eta_k}\norm{\rho(x^k)} + \norm{\eta_ke(x^k)} \quad\le\quad\norm{\rho(x^k)} + \epsilon(x^k).\hspace*{1cm}\qedhere
  \end{align*}
\end{proof}
Lemma~\ref{lem.descent} and Lemma~\ref{lem.xubounds} help prove the following crucial corollary. 
\begin{corr}
  \label{corr.main}
  Let $x^k$, $x^{k+1}$, $\eta_k$, and $c$ be as above and $k$ sufficiently large so that $c$ and $\eta_k$ satisfy~\eqref{eq.38}. Then, $\Phi(x^k)-\Phi(x^{k+1}) \ge h(x^k)$ holds with $h(x^k)$ given by 
  \begin{equation}
    \label{eq.15}
    h(x^k) := \tfrac{L^2c^3}{2(2-2Lc)}\norm{\rho(x^k)}^2 
    - \bigl(\tfrac{L^2c^2}{2-cL} + \tfrac1c\bigr)\norm{\rho(x^k)}\epsilon(x^k)
    - \bigl(\tfrac{1}{c}-\tfrac{L^2c}{2(2-cL)} \bigr)\epsilon(x^k)^2.
  \end{equation}
\end{corr}
\begin{proof}
  To simplify notation, we drop the superscripts. Thus, let $x\equiv x^k$, $u\equiv x^{k+1}$, $\eta\equiv \eta_k$. Then, we must show that
  \begin{align*}
    \Phi(x) - \Phi(u) &\quad\ge\quad  h(x)
  \end{align*}
  where $h(x)$ is given by
  \begin{equation}
    \label{eq.4}
    h(x) := a_1\norm{\rho(x)}^2 - a_2\norm{\rho(x)}\epsilon(x) - a_3\epsilon(x)^2,
  \end{equation}
  where the constants $a_1$, $a_2$, and $a_3$ are defined as
  \begin{equation}
    \label{eq.5}
    a_1 := \tfrac{L^2c^3}{2(2-2Lc)},\quad a_2 := \tfrac{L^2c^2}{2-cL} + \tfrac1c,\quad a_3 := \tfrac{1}{c}-\tfrac{L^2c}{2(2-cL)}.
  \end{equation}
  Moreover, we also note that the scalars $a_1, a_2, a_3 > 0$.
  Recall that for sufficiently large $k$, condition \eqref{eq.38} implies that
  \begin{equation}
    \label{eq.17}
    c < \eta < \tfrac{2}{L} - c\qquad\implies\qquad \tfrac{1}{\eta} > \tfrac{L}{2-Lc},\quad \tfrac{1}{\eta} < \tfrac{1}{c},\quad\text{and}\quad 2-L\eta > Lc,
  \end{equation}
  which immediately implies that
  \begin{equation*}
    \tfrac{2-L\eta}{2\eta} > \tfrac{(2-L\eta)L}{2(2-Lc)} > \tfrac{L^2c}{2(2-Lc)}\qquad\text{and}\qquad -\tfrac1\eta > -\tfrac1c.
  \end{equation*}
  Thus, we may replace~\eqref{eq.31} by the bound
  \begin{equation*}
    \Phi(x)-\Phi(u) \ge \tfrac{L^2c}{2(2-Lc)}\norm{x-u}^2-\tfrac{1}{c}\epsilon(x)\norm{x-u}.
  \end{equation*}
  Now plug in the two-sided bounds~\eqref{eq.xubounds} on $\norm{x-u}$ to obtain
  \begin{align*}
    \Phi(x)-\Phi(u) &\ge \tfrac{L^2c}{2(2-Lc)}\bigl(c\norm{\rho(x)}-\epsilon(x)\bigr)^2 
    - \tfrac{1}{c} \epsilon(x)\left(\norm{\rho(x)} + \epsilon(x)\right)\\
    &=\tfrac{L^2c^3}{2(2-Lc)}\norm{\rho(x)}^2
    - \bigl(\tfrac{L^2c^2}{2-Lc}+\tfrac{1}{c})\norm{\rho(x)}\epsilon(x)
    - \bigl(\tfrac{1}{c}-\tfrac{L^2c}{2(2-Lc)}\bigr)\epsilon(x)^2 =: h(x).
  \end{align*}
  All that remains to show is that the respective coefficients of $h(x)$ are positive. Since $2-Lc > 0$ and $c > 0$, the positivity of $a_1$ and $a_2$ is immediate. Since $0 < c < 1/L$ (see assumption~\eqref{eq.38}). Reducing inequality $a_3 = \tfrac1c- \tfrac{L^2c}{2(2-Lc)}> 0$, shows that it holds as long as $0 < c < \tfrac{\sqrt{5}-1}{L}$, which is obviously true since $c < 1/L$. Thus, the three scalars $a_1, a_2, a_3$ defined by~(\ref{eq.5}) are all positive.
\end{proof}

We now have all the ingredients to state the main convergence theorem.
\begin{theorem}[Convergence]
  \label{thm.cvg}
  Let $f \in C_L^1(\Xc)$ such that $\inf_{\Xc} f > -\infty$ and let $g$ be lsc, convex on $\Xc$. Let $\set{x^k} \subset \Xc$ be a sequence generated by~(\ref{eq.main}), and let condition~(\ref{eq.60}) on each $\norm{e(x^k)}$ hold. There exists a limit point $x^*$ of the sequence $\set{x^k}$, and a constant $K > 0$, such that $\norm{\rho(x^*)} \le K\epsilon(\sx)$. If $\set{\Phi(x^k)}$ converges, then for every limit point $x^*$ of $\set{x^k}$ it holds that $\norm{\rho(x^*)} \le K\epsilon(\sx)$. 
\end{theorem}
\begin{proof}
  Lemma~\ref{lem.descent}, \ref{lem.xubounds}, and Corollary~\ref{corr.main} have done all the hard work. Indeed, they allow us to reduce our convergence proof to the case where the analysis of the differentiable case becomes applicable, and an appeal to the analysis of~\citep[Thm.~2.1]{solodov97} grants us our claim. 
\end{proof}

Theorem~\ref{thm.cvg} says that we can obtain an approximate stationary point for which the norm of the residual is bounded by a linear function of the error level. The statement of the theorem is written in a conditional form, because nonvanishing  errors $e(x)$ prevent us from making a stronger statement. In particular, once the iterates enter a region where the residual norm falls below the error threshold, the behavior of $\set{x^k}$ may be arbitrary. This, however, is a small price to pay for having the added flexibility of nonvanishing errors. Under the stronger assumption of vanishing errors (and diminishing stepsizes), we can also obtain guarantees to exact stationary points. 

\section{Scaling up \algo: incremental variant}
\label{sec.incr}
We now apply \algo to the large-scale setting, where we have composite objectives of the form
\begin{equation}
  \label{eq.53}
  \Phi(x) := \nlsum_{t=1}^T f_t(x) + g(x),
\end{equation}
where each $f_t : \reals^n \to \reals$ is a $C_{L_t}^1(\Xc)$ function. For simplicity, we use $L = \max_t L_t$ in the sequel. It is well-known that for such decomposable objectives it can be advantageous to replace the full gradient $\sum_t \nabla f_t(x)$ by an \emph{incremental gradient} $\nabla f_{\sigma(t)}(x)$, where $\sigma(t)$ is some suitable index. 

Nonconvex incremental methods for the differentiable case been extensively analyzed in the setting of backpropagation algorithms~\citep{bertsekas10,solodov97}, which corresponds to $g(x) \equiv 0$. However, when $g(x) \neq 0$, the only incremental methods that we are aware of are stochastic generalized gradient methods of~\citep{ermoliev98} or the generalized gradient methods of~\citep{solodov98}. As previously mentioned, both of these fail to exploit the composite structure of the objective function, a disadvantage even in the convex case~\citep{duchi09a}. 

In stark contrast, we \emph{do}  exploit the composite structure of~\eqref{eq.53}. Formally, we propose the following incremental nonconvex proximal-splitting iteration:
\begin{equation}
  \label{eq.32}
  \begin{split}
    x^{k+1} &= \Mc\bigl(x^k - \eta_k\nlsum_{t=1}^T \nabla f_t(x^{k,t})\bigr),\quad k=0,1,\ldots,\\
    x^{k,1} = x^k,&\quad x^{k,t+1} = \Oc(x^{k,t} - \eta_k\nabla f_t(x^{k,t})),\quad t=1,\ldots,T-1,
  \end{split}
\end{equation}
where $\Oc$ and $\Mc$ are appropriate operators, different choices of which lead to different algorithms. For example, when $\Xc=\reals^n$, $g(x) \equiv 0$, $\Mc=\Oc=\id$, and $\eta_k \to 0$, then~\eqref{eq.32} reduces to the classic incremental gradient method (IGM)~\citep{bertsekas99}, and to the IGM of~\citep{solodov98b}, if $\lim\eta_k = \bar{\eta} > 0$. If $\Xc$ a closed convex set, $g(x) \equiv 0$, $\Mc$ is orthogonal projection onto $\Xc$, $\Oc=\id$, and $\eta_k \to 0$, then iteration \eqref{eq.32} reduces to (projected) IGM~\citep{bertsekas99,bertsekas10}. 

We may consider four variants of~\eqref{eq.32} in Table~\ref{tab.var}; to our knowledge, all of these are new. Which of the four variants one prefers depends on the complexity of the constraint set $\Xc$ and cost to apply $\pp_\eta^g$. The analysis of all four variants is similar, so we present details only for the most general case.

\begin{table}[h]\small
  \centering
  \begin{tabular}{c|c|c|c||l|l}
    $\Xc$ & $g$ & $\Mc$ & $\Oc$ & Penalty and constraints & Proximity operator calls\\
    \hline
    $\reals^n$ & $\not\equiv 0$ & $\pp_\eta^g$ & $\id$ & penalized, unconstrained & once every \emph{major} $(k)$ iteration\\
    $\reals^n$ & $\not\equiv 0$ & $\pp_\eta^g$ & $\pp_\eta^g$ & penalized, unconstrained & once every \emph{minor} $(k,t)$ iteration\\
    Convex & $h(x) + \ind{x}{\Xc}$ & $\pp_\eta^g$ & $\id$ & penalized, constrained & once every major $(k)$ iteration\\
    Convex & $h(x) + \ind{x}{\Xc}$ & $\pp_\eta^g$ & $\pp_\eta^g$ & penalized, constrained & once every minor $(k,t)$ iteration\\
    \hline
  \end{tabular}
  \caption{\small Different variants of incremental \algo~\eqref{eq.32}.}
  \label{tab.var}
\end{table}

\subsection{Convergence analysis}
\label{sec.incrConvg}
Specifically, we analyze convergence for the case $\Mc=\Oc=\pp_\eta^g$ by generalizing the differentiable case treated by~\citep{solodov98b}. We begin by rewriting~(\ref{eq.32}) in a form that matches the main iteration~(\ref{eq.main}):
\begin{equation}
  \label{eq.16}
  \begin{split}
  x^{k+1} &= \pp_\eta^g\bigl(x^k - \eta_k\nlsum_{t=1}^T \nabla f_t(x^{k,t})\bigr)\\
  &=\pp_\eta^g\bigl(x^k - \eta_k\nlsum_{t=1}^T\nabla f_t(x^k) + \eta_k\bigl[\nlsum_{t=1}^T f_t(x^k) -f_t(x^{k,t})\bigr]\bigr)\\
  &=\pp_\eta^g\bigl(x^k -\eta_k \nlsum_t\nabla f_t(x^k) + \eta_ke(x^k)\bigr).
\end{split}
\end{equation}
 %
To show that iteration~\eqref{eq.16} is well-behaved and actually fits the main \algo iteration~\eqref{eq.main}, we must ensure that the norm of the error term is bounded. We show this via a sequence of lemmas.
\begin{lemma}[Bounded-increment]
  \label{lem.zdiff}
  Let $x^{k,t+1}$ be computed by~\eqref{eq.32}, and let $s^t \in \partial g(x^{k,t})$. Then,
  \begin{equation}
    \norm{x^{k,t+1}-x^{k,t}} \le 2\eta_k\norm{\nabla f_t(x^{k,t}) + s^t}.
  \end{equation}
\end{lemma}
\begin{proof}
  From the definition of a proximity operator~\eqref{eq.7}, we have the inequality
  \begin{align*}
    &\half\norm{x^{k,t+1}-x^{k,t} + \eta_k \nabla f_t(x^{k,t})}^2 + \eta_k g(x^{k,t+1}) \le \half \norm{\eta_k \nabla f_t(x^{k,t})}^2 + \eta_k g(x^{k,t}),\\
   \hskip-12pt\implies\quad &\half\norm{x^{k,t+1}-x^{k,t}}^2 \le \eta_k\ip{\nabla f_t(x^{k,t})}{x^{k,t}-x^{k,t+1}} + \eta_k(g(x^{k,t}) - g(x^{k,t+1})).
  \end{align*}
  Since $s_t \in \partial g(x^{k,t})$, we have $g(x^{k,t+1}) \ge g(x^{k,t}) + \ip{s_t}{x^{k,t+1}-x^{k,t}}$. Therefore,
  \begin{align*}
    \half\norm{x^{k,t+1}-x^{k,t}}^2 &\le \eta_k\ip{s^t}{x^{k,t}-x^{k,t+1}} + \ip{\nabla f_t(x^{k,t})}{x^{k,t}-x^{k,t+1}}\\
    &\le \eta_k\norm{s_t+\nabla f_t(x^{k,t})} \norm{x^{k,t}-x^{k,t+1}}\\
    \implies\quad& \norm{x^{k,t+1}-x^{k,t}} \le 2\eta_k \norm{\nabla f_t(x^{k,t}) + s^t}.\qedhere
  \end{align*}
\end{proof}
Lemma~\ref{lem.zdiff} proves helpful in bounding the overall error.
\begin{lemma}[Incrementality error]
  \label{lem.perstep}
  To ease notation, define $z^t = x^{k,t}$, $x = x^k$, and $\eta=\eta_k$. Define
  \begin{equation}
    \label{eq.29}
    \epsilon_t := \norm{\nabla f_t(z^t) - \nabla f_t(x)},\quad t=1,\ldots,T.
  \end{equation}
  Then, for each $t \ge 2$, the following bound on the error holds:
  \begin{equation}
    \label{eq.43}
    \epsilon_t \le 2\eta L \nlsum_{j=1}^{t-1}(1+2\eta L)^{t-1-j}\norm{\nabla f_j(x) + s^j},\quad t=2,\ldots,T.
  \end{equation}
\end{lemma}
\begin{proof}
  Our proof extends the differentiable setting of~\citep{solodov98b} to our nondifferentiable setting. We proceed by induction. The base case is $t=2$, for which we have
  \begin{align*}
    \epsilon_2 = \norm{\nabla f_2(z^2)-\nabla f_2(x)} \le L\norm{z^2-x} = L\norm{z^2-z^1} \le 2\eta L\norm{\nabla f_1(x)+s^1},
  \end{align*}
  where the last inequality follows from Lemma~\ref{lem.zdiff}. Now assume inductively that~(\ref{eq.43}) holds for $t \le r < T$, and consider $t=r+1$. In this case we have 
  \begin{alignat}{3}
    \nonumber
    \epsilon_{r+1} &\quad=\quad\norm{\nabla f_{r+1}(z^{r+1}) - \nabla f_{r+1}(x)}&\quad\le\quad& L\norm{z^{r+1}-x}\\
    \nonumber
    &\quad=\quad L\bignorm{\nlsum_{j=1}^r(z^{j+1}-z^j)}{}&\quad\le\quad& L\nlsum_{j=1}^r\norm{z^{j+1}-z^j}\\
    \label{eq.20}
    &\quad\le\quad 2\eta L\nlsum_{j=1}^r\norm{\nabla f_j(z^j) + s^j},
  \end{alignat}
  where the last inequality again follows from Lemma~\ref{lem.zdiff}.

  To complete the induction, first observe that $\norm{\nabla f_t(z^t)+s^t} \le \norm{\nabla f_t(x)+s^t} + \epsilon_t$. Thus, invoking the induction hypothesis, we obtain
  \begin{equation}
    \label{eq.19}
    \norm{\nabla f_t(z^t)} \le \norm{\nabla f_t(x)} + 2\eta L \nlsum_{j=1}^{t-1}(1+2\eta L)^{t-1-j}\norm{\nabla f_j(x) + s^j},\quad t=2,\ldots,r.
  \end{equation}
  Combining inequality~\eqref{eq.19} with~\eqref{eq.20} we further obtain
  \begin{equation*}
    \epsilon_{r+1} \le 2\eta L\nlsum_{j=1}^r\left(\norm{\nabla f_j(x)+s^j} + 2\eta L\nlsum_{l=1}^{j-1}(1+L\eta)^{j-1-l}\norm{\nabla f_l(x)+s^l}\right).
  \end{equation*}
  Introducing the shorthand $\beta_j \equiv \norm{\nabla f_j(x)+s^j}$, simple manipulation of the above inequality yields
  \begin{align*}
    \epsilon_{r+1}&\le\quad2\eta L \beta_r + \nlsum_{l=1}^{r-1} \left( 2\eta L + 4\eta^2L^2\nlsum_{j=l+1}^{r}(1+2\eta L)^{j-l-1}\right)\beta_l\\
    &=\quad 2\eta L \beta_r + \nlsum_{l=1}^{r-1} \left(2\eta L + 4\eta^2L^2\nlsum_{j=0}^{r-l-1}(1+2\eta L)^{j}) \right)\beta_l\\
    &=\quad 2\eta L \beta_r + \nlsum_{l=1}^{r-1} 2\eta L(1+2\eta L)^{r-l}\beta_l\quad=\quad 2\eta L\nlsum_{l=1}^{r}(1+2\eta L)^{r-l}\beta_l,
  \end{align*}
  which completes the proof.
\end{proof}

\begin{lemma}[Bounded error]
  \label{lem.errbnd}
  If for all $x^k \in \Xc$, $\norm{\nabla f_t(x^k)} \le M$ and $\norm{\partial g(x^k)} \le G$, then there exists a constant $K_1 > 0$ such that 
  $\norm{e(x^k)} \le K_1$.
\end{lemma}
\begin{proof}
  To bound the error of using $x^{k,t}$ instead of $x^k$ first define the term
  \begin{equation}
    \label{eq.17}
    \epsilon_t := \norm{\nabla f_t(x^{k,t}) - \nabla f_t(x^k)},\quad t=1,\ldots,T.
  \end{equation}
  Then, Lemma~\ref{lem.perstep} shows that for \fromto{2}{t}{T}
  \begin{equation}
    \label{eq.43}
    \epsilon_t \le 2\eta_k L \nlsum_{j=1}^{t-1}(1+2\eta_k L)^{t-1-j}\norm{\nabla f_j(x^k) + s^j}.
  \end{equation}
  Since $\norm{e(x^k)}=\sum_{t=1}^T \epsilon_t$, and $\epsilon_1=0$, \eqref{eq.43} then leads to the bound
    \begin{align*}
    \nlsum_{t=2}^T\epsilon_t&\le 2\eta_k L\nlsum_{t=2}^T\nlsum_{j=1}^{t-1}(1+2\eta_k L)^{t-1-j}\beta_j =   2\eta_k L\nlsum_{t=1}^{T-1} \beta_t \Bigl( \nlsum_{j=0}^{T-t-1} (1+2\eta_k L)^j \Bigr)\\
    &\le \nlsum_{t=1}^{T-1}(1+2\eta_k L)^{T-t}\beta_t\quad\le\quad(1+2\eta_k L)^{T-1}\nlsum_{t=1}^{T-1}\norm{\nabla f_t(x) + s^t}\\
    &\le C_1(T-1)(M+G) =: K_1.\hspace*{1cm}\qedhere
  \end{align*}
\end{proof}
Thus, the error norm $\norm{e(x^k)}$ is bounded from above by a constant, whereby it satisfies the requirement~(\ref{eq.60}), making the incremental \algo method (\ref{eq.32}) a special case of the general \algo framework. This allows us to invoke the convergence result Theorem~\ref{thm.cvg} for without further ado.

\section{Illustrative application}
\label{sec.app}
The main contribution of our paper is the new \algo framework, and a specific application is not one of the prime aims of this paper. We do, however, provide an illustrative application of \algo to a challenging nonconvex problem: \emph{sparsity regularized low-rank matrix factorization}
\begin{equation}
  \label{eq.48}
  \min_{X, A \ge 0}\quad\half\frob{Y-XA}^2 + \psi_0(X) + \nlsum_{t=1}^T\psi_t(a_t),
\end{equation}
where $Y \in \reals^{m\times T}$, $X \in \reals^{m\times K}$ and $A \in \reals^{K\times T}$, with $a_1,\ldots,a_T$ as its columns. Problem~\eqref{eq.48} generalizes the well-known nonnegative matrix factorization (NMF) problem of~\citep{lee00} by permitting arbitrary $Y$ (not necessarily nonnegative), and adding regularizers on $X$ and $A$. A related class of problems was studied in~\citep{mairal10a}, but with a crucial difference: the formulation in~\citep{mairal10a} \emph{does not} allow nonsmooth regularizers on $X$. The class of problems studied in~\citep{mairal10a} is in fact a subset of those covered by \algo. On a more theoretical note, \citep{mairal10a} considered stochastic-gradient like methods whose analysis requires computational errors and stepsizes to vanish, whereas our method is deterministic and allows nonvanishing stepsizes and errors.

Following~\citep{mairal10a} we also rewrite~\eqref{eq.48} in a form more amenable to \algo. We eliminate $A$ and consider
\begin{equation}
  \label{eq.49}
  \nlmin_{X}\quad \phi(X) := \nlsum_{t=1}^Tf_t(X) + g(X),\quad\text{where}\quad g(X):=\psi_0(X) + \ind{\ge 0}{X},
\end{equation}
and where each $f_t(X)$ for \fromto{1}{t}{T} is defined as
\begin{equation}
  \label{eq.50}
  f_t(X) := \nlmin_{a}\quad\half\norm{y_t - Xa}^2 + g_t(a),
\end{equation}
where $g_t(a) := \psi_t(a) + \ind{\ge 0}{a}$. For simplicity, assume that~\eqref{eq.50} attains its unique\footnote{Otherwise, at the expense of more notation, we can add a small strictly convex perturbation to ensure uniqueness; this perturbation can be then absorbed into the overall computational error.} minimum, say $a^*$, then $f_t(X)$ is differentiable and we have $\nabla_X f_t(X) = (Xa^*-y_t)(a^*)^T$. Thus, we can instantiate~(\ref{eq.32}), and all we need is a subroutine for solving~\eqref{eq.50}.\footnote{In practice, it is better to use \emph{mini-batches}, and we used the same sized mini-batches for all the algorithms.}

We present empirical results on the following two variants of~\eqref{eq.49}: (i) pure unpenalized NMF ($\psi_t\equiv 0$ for \fromto{0}{t}{T}) as a baseline; and (ii) sparsity penalized NMF where $\psi_0(X)\equiv\lambda\norml{X}$ and $\psi_t(a_t) \equiv \gamma\norml{a_t}$.  
Note that without the nonnegativity constraints, \eqref{eq.49} is similar to  sparse-PCA.

We use the following datasets and parameters: 
\begin{inparaenum}[$\llbracket$i$\rrbracket$]
\setlength{\itemsep}{-1pt}
\item \texttt{RAND}: $4000 \times 4000$ dense random (uniform $[0,1]$); rank-32 factorization; $(\lambda,\gamma)=(10^{-5},10)$;
\item \texttt{CBCL}: CBCL database~\citep{sung96}; $361 \times 2429$; rank-49 factorization; 
\item \texttt{YALE}: Yale B Database~\citep{kclee05}; $32256 \times 2414$ matrix; rank-32 factorization; 
\item \texttt{WEB:} Web graph from google; sparse $714545 \times 739454$ (empty rows and columns removed) matrix; ID: 2301 in the sparse matrix collection~\citep{davis11}); rank-4 factorization; $(\lambda=\gamma=10^{-6})$.
\end{inparaenum}

\begin{figure}[tbp]
  \centering
  \begin{tabular}{ccc}
    \hskip-12pt\includegraphics[width=0.33\linewidth]{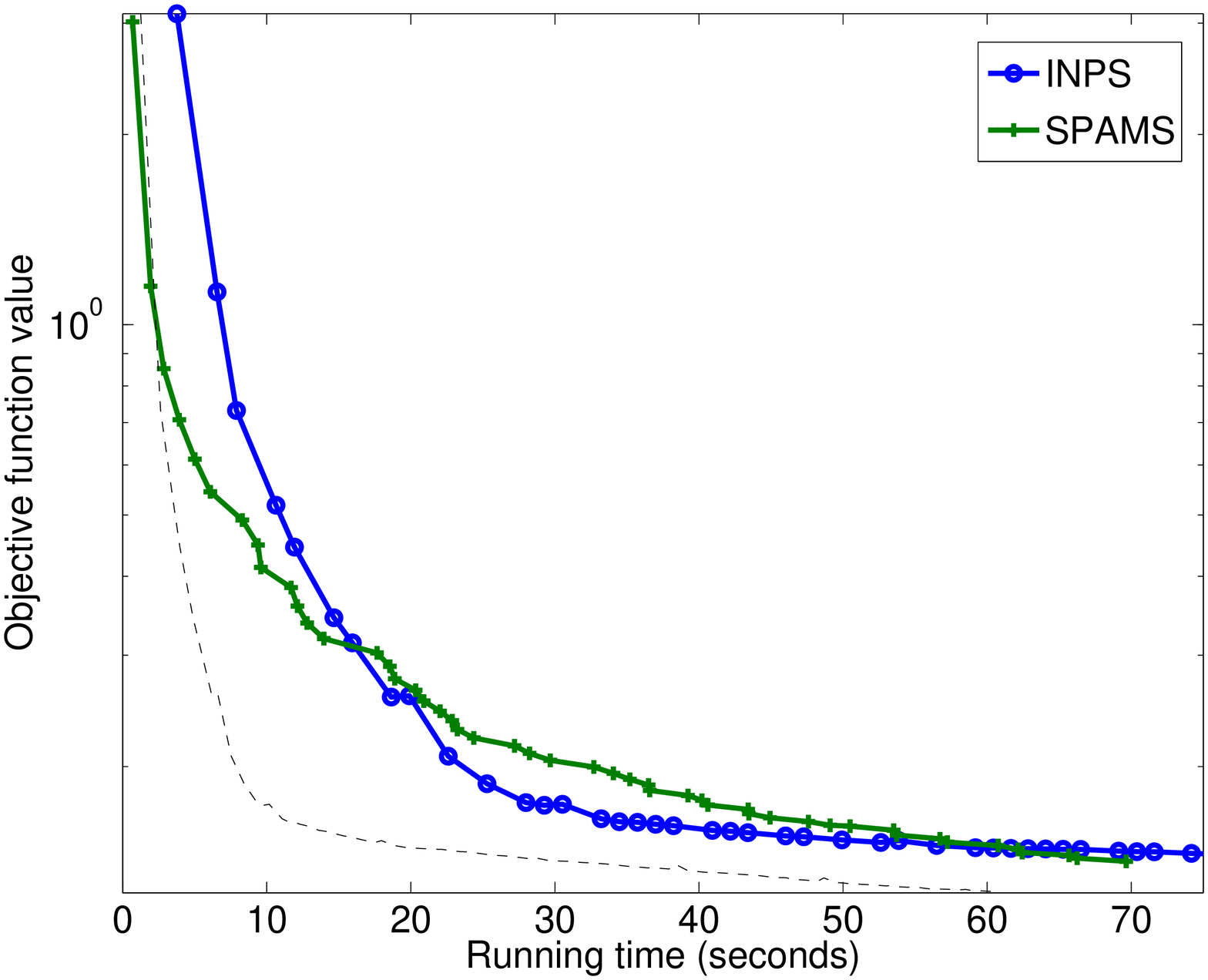} &
    \hskip-12pt\includegraphics[width=0.33\linewidth]{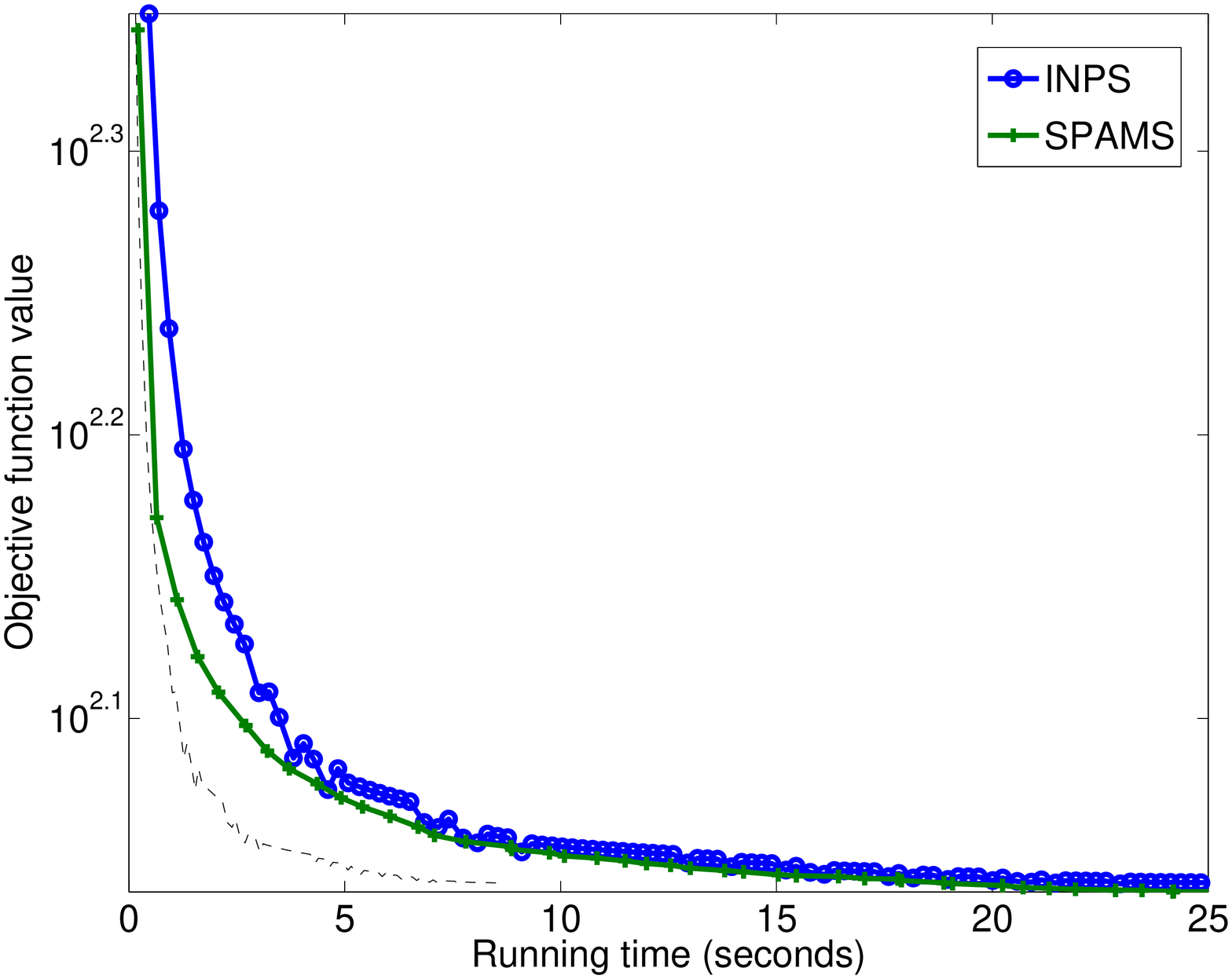} &
    \hskip-12pt\includegraphics[width=0.33\linewidth]{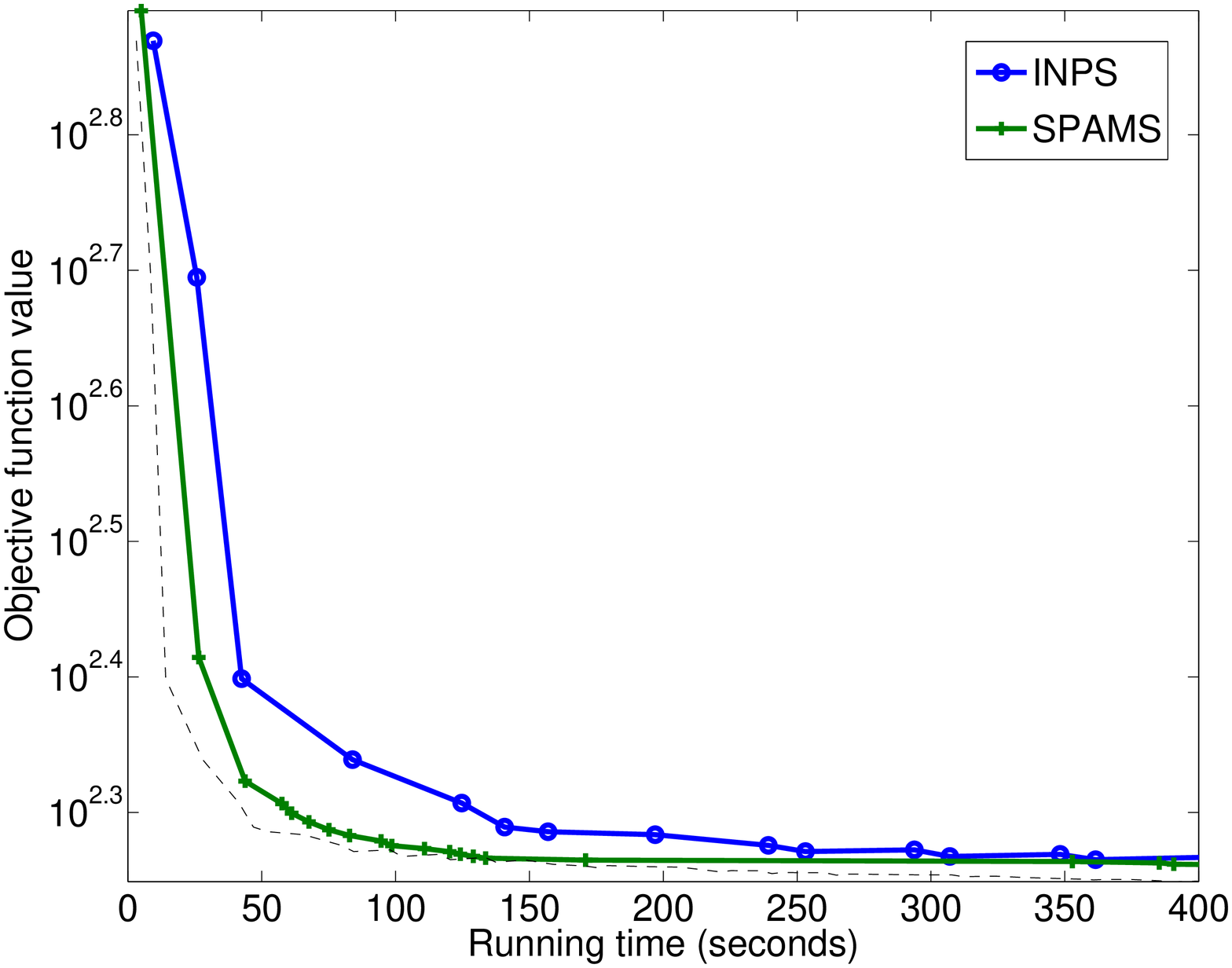}\\
\end{tabular}
  \caption{\small Running times of \algo (\text{Matlab}) versus SPAMS (C++) for NMF on \texttt{RAND}, \texttt{CBCL}, and \texttt{YALE} datasets.  Initial objective values and tiny runtimes have been suppressed for clarity of presentation.}
  \label{fig.res}
\end{figure}
On the NMF baseline (Fig.~\ref{fig.res}), we compare \algo against the well optimized state-of-the-art C++ toolbox SPAMS (version 2.3)~\citep{mairal10a}. We compare against SPAMS only on dense matrices, as its NMF code seems to be optimized for this case. Obviously, the comparison is not fair: unlike SPAMS, \algo and  its subroutines are all implemented in \textsc{Matlab}, and they run equally easily on large sparse matrices. Nevertheless, \algo proves to be quite competitive: Fig.~\ref{fig.res} shows that our \textsc{Matlab} implementation runs only slightly slower than SPAMS. We expect a well-tuned C++ implementation of \algo to run at least 4--10 times faster than the \textsc{Matlab} version---the dashed line in the plots visualizes what such a mere 3X-speedup to \algo might mean.

Figure~\ref{fig.snmf} shows numerical results comparing the stochastic generalized gradient (SGGD) algorithm of~\citep{ermoliev98} against \algo, when started at the same point. As in well-known, SGGD requires careful stepsize tuning; so we searched over a range of stepsizes, and have reported the best results. \algo too requires some stepsize tuning, but substantially lesser than SGGD. As predicted, the solutions returned by \algo have objective function values lower than SGGD, and have greater sparsity.

\begin{figure}[htbp]
  \centering
  \begin{tabular}{cc|cc}
    \hskip-15pt\includegraphics[width=0.33\linewidth]{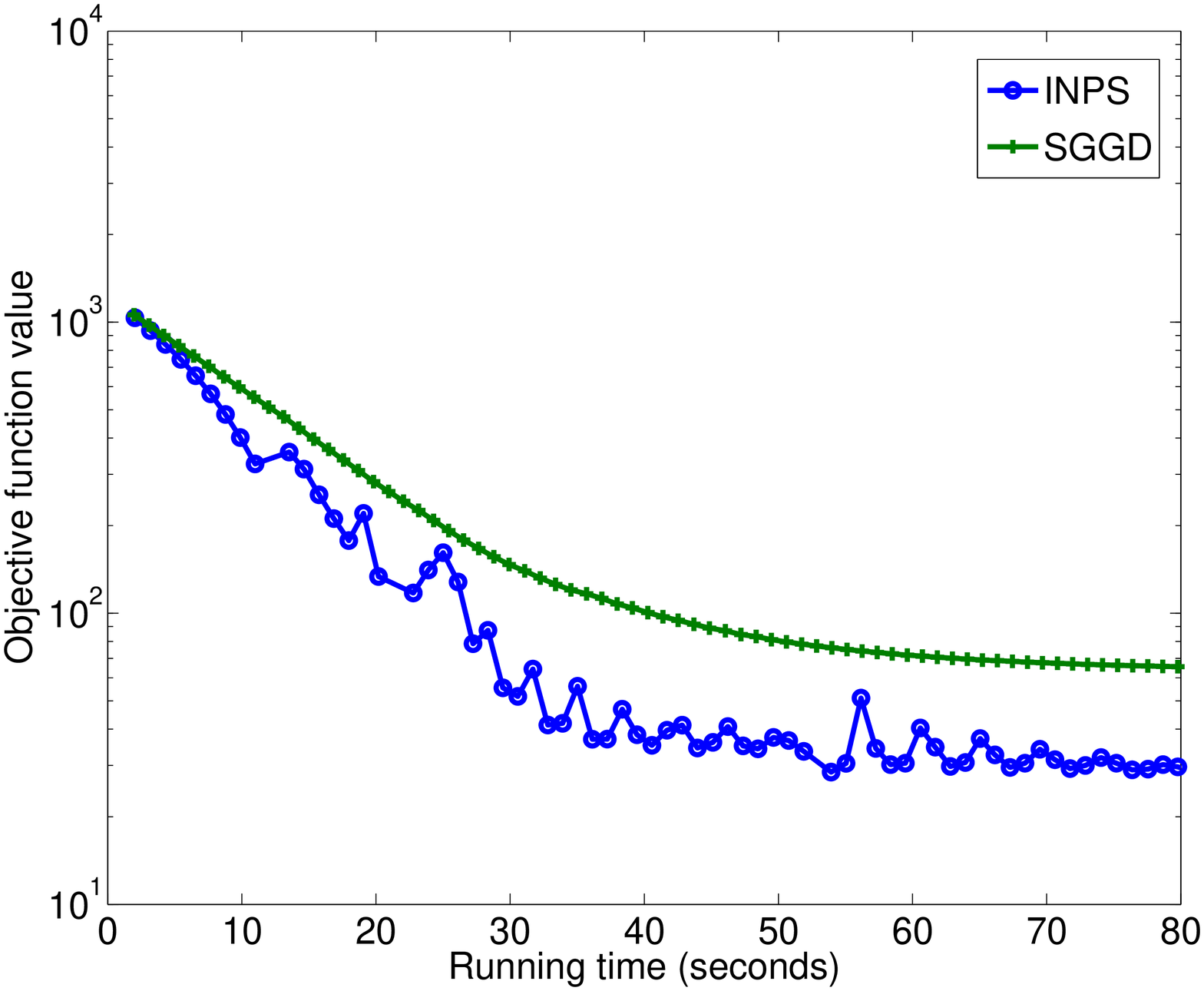}&
    \hskip-22pt\raisebox{28pt}{\includegraphics[width=0.23\linewidth]{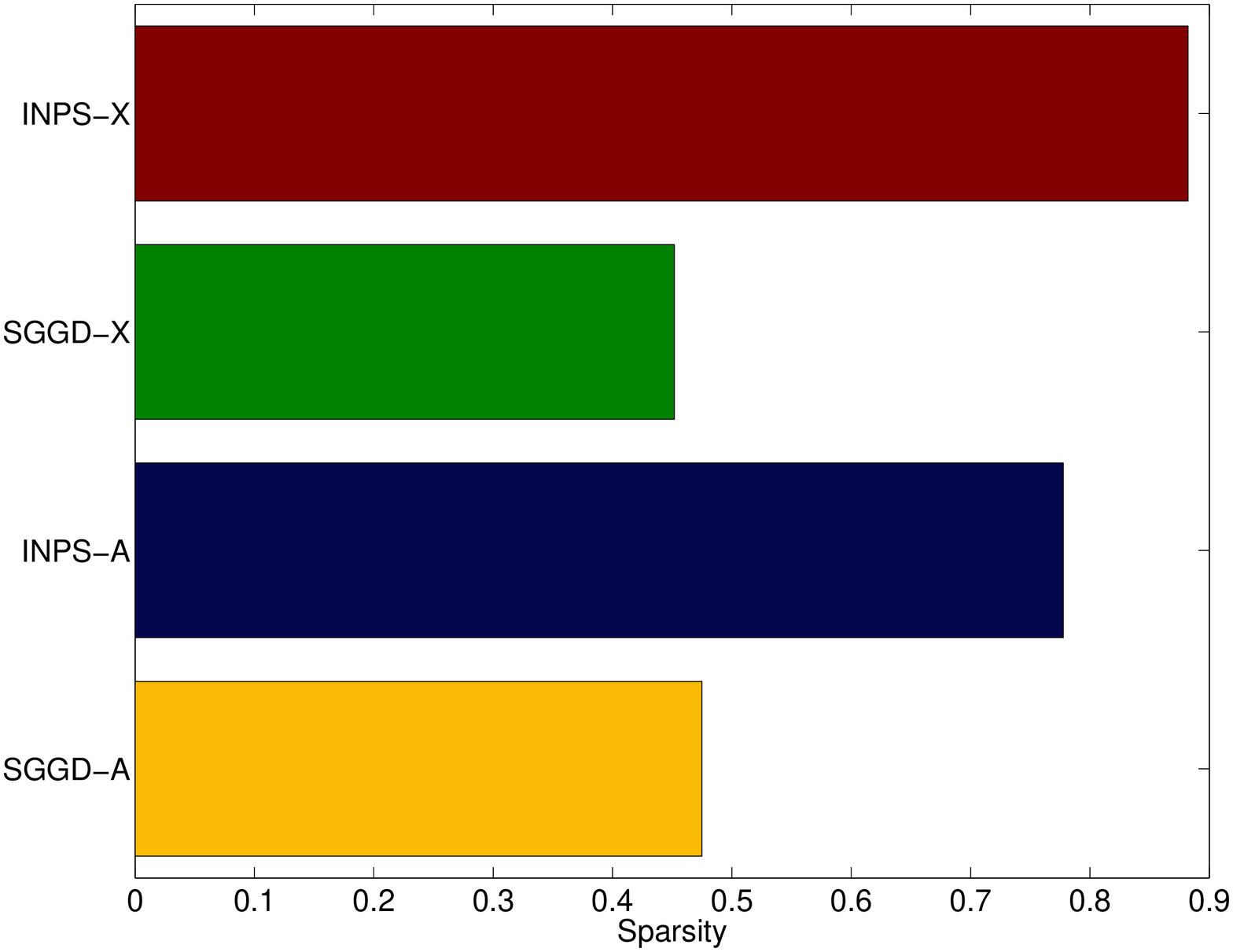}}&
\end{tabular}
  \caption{\small Sparse NMF: \algo versus SGGD. The bar plots show the sparsity (higher is better) of the factors $X$ and $A$. Left plots for \texttt{RAND} dataset; right plots for~\texttt{WEB}. As expected, SGGD yields slightly worse objective function values and less sparse solutions than \algo.}
  \label{fig.snmf}
\end{figure}

\section{Discussion}
\label{sec.disc}
We presented a new framework called \algo, which solves a broad class of nonconvex composite objective problems. 
\algo permits nonvanishing computational errors, which can be practically useful. We specialized \algo to also obtain a scalable incremental version. Our numerical experiments on large scale matrix factorization indicate that \algo is competitive with state-of-the-art methods.

We conclude by mentioning that \algo includes numerous other algorithms as special cases. For example, batch and incremental convex FBS, convex and nonconvex gradient projection, the proximal-point algorithm, among others. Theoretically, however, the most exciting open problem resulting from this paper is: \emph{extend \algo in a scalable way when even the nonsmooth part is nonconvex}. This case will require very different convergence analysis, and is left to the future.


\setlength{\bibsep}{2pt}
\bibliographystyle{abbrvnat}

\appendix
\section{Implementation notes}
If $\Xc$ is the nonnegative orthant $ \reals_+^n$, then the proximity operator $\prox_\eta^g$ often simplifies as
\begin{equation}
  \label{eq.55}
  \prox_\eta^g(y) = \prox_\eta^\psi(\proj_{\Xc}y).
\end{equation}
Additionally, if $\psi$ is an elementwise separable function, then one can easily admit a box-plus-hyperplane constraint set $\Xc$ of the form
\begin{equation}
  \label{eq.56}
  \Xc = \set{ x \in \reals^n\ |\  l_i \le x_i \le u_i, \text{for } 1 \le i \le n,\ \text{and } a^Tx=b}.
\end{equation}
For more general constraint sets, we can invoke \emph{Dykstra splitting}~\citep{combettes10}, which solves the problem
\begin{equation}
  \label{eq.57}
  \min\quad\half\norm{x-y}^2 + \psi(x) + \delta(x | \Xc),
\end{equation}
by using the following algorithm
\begin{center}
\framebox{
  \begin{minipage}{0.55\linewidth}
    \small{\it Dykstra splitting for~\eqref{eq.57}}\\[-2mm]
    \hrule
    \vskip 3pt
    Initialize $x\gets y$, $p\gets 0$, $q\gets 0$\\
    While $\neg$ converged, iterate:\\[-2mm]
    \begin{equation}
      \label{eq.58}
      \hskip -4cm
      \left\lfloor
        \begin{array}{l@{\hspace{2pt}}l}
          e &\gets\prox_1^\psi(x + p)\\
          p &\gets x+p-e\\
          x &\gets\proj_{\Xc}(e + q)\\
          q &\gets e+q-x.
        \end{array}
      \right.
    \end{equation}
  \end{minipage}
}
\end{center}
It can be shown that iterating~\eqref{eq.58} converges to the solution of~\eqref{eq.57}. In practice, it usually suffices to run Dykstra splitting for a few iterations (2--10) only. 

\end{document}